\documentclass[11pt]{amsart}
\usepackage{amsmath,amssymb,mathrsfs,xcolor,comment}
\usepackage{enumitem}
\usepackage{hyperref}
\usepackage{graphicx}
\hypersetup{
    colorlinks=true, %set true if you want colored links
    linktoc=all,     %set to all if you want both sections and subsections linked
    linkcolor=blue,  %choose some color if you want links to stand out
}

\newtheorem{theorem}{Theorem}[section]
\newtheorem{proposition}[theorem]{Proposition}
\newtheorem{corollary}[theorem]{Corollary}
\newtheorem{lemma}[theorem]{Lemma}

\theoremstyle{definition}
\newtheorem{definition}[theorem]{Definition}
\newtheorem{remark}[theorem]{Remark}

\numberwithin{equation}{section}

\newcommand{\pr}{\partial}
\renewcommand\div{\operatorname{div}}

\newcommand{\diam}{\operatorname{diam}}
\newcommand{\curv}{\kappa}
\DeclareMathOperator{\sn}{sn}
\DeclareMathOperator{\cs}{cs}
\DeclareMathOperator{\ct}{ct}
\DeclareMathOperator{\tn}{tn}

\begin{document}

\title[Moving monotonicity formulae in constant curvature]{Moving monotonicity formulae for minimal submanifolds in constant curvature}
\author{Keaton Naff}
\address{Department of Mathematics, Massachusetts Institute of Technology, Cambridge, MA 02139, USA}
\email{kn2402@mit.edu}

\author{Jonathan J. Zhu}
\address{Department of Mathematics, University of Washington, Seattle, WA, USA}
\email{jonozhu@uw.edu}

\begin{abstract}
We discover new monotonicity formulae for minimal submanifolds in space forms, which imply the sharp area bound for minimal submanifolds through a prescribed point in a geodesic ball. These monotonicity formulae involve an energy-like integral over sets which are, in general, not geodesic balls. In the Euclidean case, these sets reduce to the moving-centre balls introduced by the second author in \cite{Zhu18}. 
\end{abstract}
\maketitle
%\tableofcontents

\section{Introduction}
 
Recently, in \cite{NZ}, we studied an area estimate for minimal submanifolds in geodesic balls in space forms which pass through a prescribed point. Such area estimates sometimes follow from a suitable monotonicity formula; monotonicity formulae naturally encapsulate strictly more information and are often used to give more precise control of geometric quantities. In the following brief note, we will exhibit some new monotonicity formulae for minimal submanifolds in space forms. These formulae can be used to recover the sharp area estimates proved in \cite{NZ} and \cite{BH17}.

Consider a space form $M \in \{\mathbb{H}^n, \mathbb{R}^n, \mathbb{S}^n\}$, and a geodesic ball $B^n_R$ in $M$ with radius $R \in (0, \frac{1}{2} \mathrm{diam}(M))$ and centre $o$. Define $A(r) := \int_0^r \sn(t)^{k-1} dt$, where $g=dr^2 + \sn(r)^2 g_{\mathbb{S}^{n-1}}$ and $\sn(r)$ is the usual warping function. Given a point $y$ in $B^n_R$, let $\gamma \subset M$ be the maximal geodesic containing $o$ and $y$. There is a foliation of $B^n_R$ by totally geodesic $(n-1)$-dimensional disks $\Gamma_s \subset B^n_R$, $s \in (-R, R)$, which meet $\gamma$ orthogonally.\footnote{In $\mathbb{R}^n$, for instance, these are just the intersection of hyperplanes orthogonal to $y$ with $B^n_R$.} Let $r_y$ be the distance function from the point $y$. There is a unique function $u_s$ on $B^n_R$ which agrees with $r_y$ on $\pr B^n_R$ and such that $u_s = u(s)$ is constant on each $\Gamma_s$ (see Sections 2.4 and 3.2 in \cite{NZ} for more details and Definition \ref{eq:def-u} below).

\begin{theorem}
\label{thm:main}
Suppose $M \in \{\mathbb{H}^n, \mathbb{R}^n, \mathbb{S}^n\}$ is a space form, $R \in (0, \frac{1}{2} \mathrm{diam}(M))$, and $y \in B^n_R$. If $M= \mathbb{S}^n$, further assume that $\cos(R + d(0,y))^2 \geq \frac{2}{k}$. Suppose that $\Sigma$ is a $k$-dimensional minimal submanifold in $B^n_R$ with $\pr\Sigma \subset \pr B^n_R$. 

Define $f:= \frac{A(r_y)}{A(u_s)}$ and let $E_t := \{f\leq t\}$. Let $\Sigma_0$ be any totally geodesic $k$-disk orthogonal to the geodesic $\gamma$ containing $o$ and $y$. There is a continuous family of functions $G_t\geq 0$ on $B^n_R$, with $G_0=0$, so that the quantity 
\[ Q(t) := \frac{\int_{\Sigma \cap E_t} \left( |\nabla^\top r_y|^2 - G_t  |\nabla^\top u|^2 \right) }{|\Sigma_0\cap E_t|}\] 
is monotone increasing for $t\in[0,1]$. Moreover, $Q$ is constant if and only if $\Sigma$ is a totally geodesic $k$-disk orthogonal to $\gamma$. 
\end{theorem}

For a more precise statement, the reader may consult Theorem \ref{thm:weighted-monotonicity}, which also includes certain other monotonicity formulae for $M\in \{\mathbb{H}^n,\mathbb{R}^n\}$ (see also Remark \ref{rmk:F'}). One of these is equivalent to the moving-centre monotonicity formula previously found by the second author \cite{Zhu18}. This latter monotonicity formula is a proper (unweighted) area monotonicity, whereas all of the new monotonicity formulae in this paper require a weight similar to the one in Theorem \ref{thm:main}. We note that even in the classical setting, $y = o$, it is not known whether an area-monotonicity holds in the sphere. Nevertheless, there is a very natural weighted monotonicity (Theorem \ref{thm:classic-weighted-monotonicity}) that holds for every space form. Theorem \ref{thm:main} can be thought of as a generalisation of this classical weighted monotonicity. There is also a related boundary monotonicity that holds in the classical setting (Theorem \ref{thm:bdry-monotonicity}), but we do not know if an analogue holds here. We remark that there seem to be fewer settings in extrinsic geometry where monotonicity formulae are known to hold, compared with the intrinsic setting (where, for instance, versions of the Bishop-Gromov monotonicity are known to hold under a variety of general settings).

In all of our monotonicity formulae, the integrand in the numerator is always bounded above by 1, and converges to 1 as $f\to 0$. For small $t$ the sublevel sets $E_t$ approximate geodesic balls around $y$. When $M \in \{\mathbb{H}^n, \mathbb{R}^n\}$, we have $E_1 = B^n_R$. Thus, one may deduce the sharp area bound $|\Sigma| \geq |\Sigma_0 \cap B^n_R|$ if $\Sigma$ passes through the prescribed point $y$ (see Corollary \ref{cor:area-estimate}). 

The sets $E_t$ behave slightly differently if $M = \mathbb{S}^n$ and in general we only have $E_1 = B^n_R \cap B^n_{\frac{\pi}{2}}(y)$. The domain $B^n_R \cap B^n_{\frac{\pi}{2}}(y)$ also appeared as an obstruction to the vector field approach used in \cite{NZ} to prove the sharp area estimate in $\mathbb{S}^n$ for certain values of $R$ and $y$ (see Section 5 there for more discussion). In our main theorem, we are only able to establish the monotonicity of $Q(t)$  if $\cos(R + d(o, y))^2\geq\frac{2}{k}$ and this implies that $B^n_R \subset B^n_{\frac{\pi}{2}}(y)$, hence $E_1 = B^n_R$. Note that this condition means the monotonicity does not hold for any $R$ and $y$ if $k = 2$. When the monotonicity does hold, we can also deduce the sharp area estimate in the sphere. We note that the direct vector field method in \cite{NZ} gave the sharp area estimate under somewhat more general conditions on $R,y$.

The proof of Theorem \ref{thm:main} originates from ideas from \cite{Zhu18}, but with two key conceptual realisations: that the moving-centre balls should be replaced by a suitable family of sublevel sets, and that the area should be replaced by a suitable weight. These developments give rise to several conditions that needs to be delicately balanced against each other. First, we extend the notion of moving-centre balls to the sublevel sets of $f = \frac{A(r_y)}{A(u_s)}$. This is motivated by the vector field used to solve the prescribed point problem in \cite{NZ} as well as the fact that it agrees with the known case in \cite{Zhu18}. Partly motivated by the classical weighted monotonicity which holds in all space forms, we also introduce a weight $w_t$, which we additionally allow to depend on $t$. A monotonicity for the energy-like quantity $\int_{\Sigma\cap E_t} w_t$ follows so long as one can find a family of vector fields $W_t$ satisfying:
\begin{enumerate}
\item (boundary condition) $\langle W_t, \frac{\nabla^\top f}{f} \rangle \leq w_t$ on $\Sigma\cap \pr E_t$
\item (divergence condition) $\mathrm{div}_{\Sigma}(W_t) \geq w_t - t \partial_t w_t$ on $\Sigma\cap E_t$. 
\end{enumerate}
(See Lemma \ref{lem:weighted-monotonicity}). In order to deduce the desired sharp area estimates from this monotonicity, the weight $w_t$ and the sets $E_t$ must satisfy a number of additional constraints. Most importantly, we should have that:
\begin{itemize}
\item $w_t\leq 1$ (with equality on the totally geodesic disks orthogonal to $\gamma$);
\item $w_t \to 1$ as $t\to 0$;
\item $E_1 = B^n_R$;
\item $E_t$ approximates balls about $y$ as $t\to 0$. 
\end{itemize}

One of the main difficulties in discovering a suitable monotonicity formula is to \textit{simultaneously} produce the three families $E_t, w_t, W_t$, which satisfy the interdependent conditions above.

\subsection*{Acknowledgements}

KN was supported by the National Science Foundation under grant DMS-2103265. 

\section{Preliminaries}

In this section, we will very briefly review some preliminaries necessary for the proof of the new monotonicity formulae. The reader can more details in \cite{NZ}. 

In what follows, $M \in \{\mathbb{H}^n, \mathbb{R}^n, \mathbb{S}^n\}$ is one of the space forms. We let $B_t^n = B_t^n(o)$ denote the geodesic ball of radius $t$ around a fixed point $o \in M$, which we call the origin. We fix some $R \in (0, \frac{1}{2} \mathrm{diam}(M))$ and some $y \in B^n_R$, and let $\Sigma \subset B_R^n$ denote a $k$-dimensional (smooth) minimal submanifold which passes through $y$ and satisfies $\partial \Sigma \subset \partial B_R^n$. Define 
\begin{equation}\label{def:sn}
\sn(r) = \begin{cases} \sinh(r),& M= \mathbb{H}^n\\  r,& M= \mathbb{R}^n\\ \sin(r),& M= \mathbb{S}^n\end{cases}.
\end{equation}
 and set $\cs(r) := \sn'(r)$, as well as $\tn(r) := \sn(r)/\cs(r)$ and $\ct(r) =1/\tn(r)$.

Given $z \in M$, we introduce the shorthand  $r_z(x) = d(x, z)$ for the distance function on $M$. Away from $z$ and the cut locus of $z$ the function $r_z$ is smooth. For our origin $o \in M$, we will write $r(x)$ in place of $r_o(x)$. 

As in \cite{NZ} we define a radius $\underline{r}(y)$ by 
\begin{equation}\label{underline-r}
\underline{r}(y) := \begin{cases} \cs^{-1}\big(\frac{\cs(R)}{\cs(r(y))}\big) ,& M \in \{\mathbb{H}^n, \mathbb{S}^n\} \\ (R^2 - r(y)^2)^\frac{1}{2} ,& M = \mathbb{R}^n \end{cases}.
\end{equation}
If $r(y)=0$, we understand that $\underline{r}(y)=R$. 

For $r \in [0, \frac{1}{2}\mathrm{diam}(M))$, we recall from the introduction the definition 
\begin{equation}
A(r) := \int_0^r \mathrm{sn}(t)^{k-1} \, dt. 
\end{equation}
Observe that $A(r)$ is positive and increasing. Moreover, \begin{equation}\frac{A''(r)}{A'(r)} = (k-1)\ct(r).\end{equation} The area of a $k$-dimensional totally geodesic disk in the space form $M$ is 
\begin{equation}
|B^k_r| = A(r) |\mathbb{S}^{k-1}|. 
\end{equation}
Finally, we note that as $r \to 0$, these functions have the asymptotics
\begin{equation}\label{asymptotics-1}
A'(r) = r^{k-1} + o(r^{k-1}), \qquad A(r) = \frac{1}{k} r^k + o(r^{k}).
\end{equation}

\subsection{The functions $s$ and $\rho$}
As in the introduction, let $\gamma \subset M$ be the unique maximal geodesic containing the points $o$ and $y$. Let $\rho(x) :=\inf_{z \in \gamma} d(x, z)$ denote the distance to the geodesic $\gamma$. For each point $x$ with $\rho(x)<\frac{1}{2}\diam(M)$, there exists a unique point $z_x \in \gamma$ such that $\rho(x) = d(x, z_x)$. Note that $\diam(M)<\infty$ only when $M=\mathbb{S}^n$. In this case $\{\rho(x) = \frac{1}{2}\diam(M)\}$ consists of a copy of $\mathbb{S}^{n-2}$ and we let $o'$ denote the antipodal point to $o$. For $M = \mathbb{S}^n$, we set $\mathcal{E} = \{\rho = \frac{1}{2} \mathrm{diam}(M)\} \cup \{x \in M : \rho(x) < \frac{1}{2} \mathrm{diam}(M) \text{ and } z_x = o'\}$.  Otherwise, we take $\mathcal{E}=\emptyset$. Note in particular that $B^n_R \subset M\setminus \mathcal{E}$ whenever $R< \frac{1}{2}\diam(M)$. Now, define $\mathrm{sign} :  \gamma \setminus \mathcal{E} \to \{-1, 0, 1\}$ by $\mathrm{sign}(z) = 1$ if $z \in \gamma\setminus \mathcal{E}$ lies on the same side of $o$ as $y$, $\mathrm{sign}(o) = 0$, and $\mathrm{sign}(z) = -1$ otherwise.  

We define a function $s: M \setminus \mathcal{E} \to (-\diam(M), \diam(M))$ by setting $s(x) = \mathrm{sign}(z)r(z_x)$. In particular, the Pythagorean theorem applied to the right geodesic triangle $oz_x x$ gives
\begin{equation}
\label{eq:pythag-0}
\begin{cases} \cs(s(x))\cs(\rho(x)) = \cs(r(x)) & M \in \{\mathbb{H}^n, \mathbb{S}^n\} \\ s(x)^2 + \rho(x)^2 = r(x)^2 & M = \mathbb{R}^n \end{cases}, 
\end{equation}
for $x \in B^n_R$. 

The function $s$ is smooth on $M\setminus \mathcal{E}$ and the function $\rho$ is a distance function on $M$. The level sets of $s$ are totally geodesic hypersurfaces in $M$; in particular $s$ is constant on any $k$-dimensional totally geodesic disk in $B^n_R$ which intersects $\gamma$ orthogonally. Indeed, the metric on $M\setminus (\mathcal{E}\cup \gamma) \cong (0, \frac{1}{2}\mathrm{diam}(M)) \times (-\diam(M), \diam(M)) \times \mathbb{S}^{n-2}$ may be written \begin{equation}\label{eq:g-s-rho}g= d\rho^2 + \cs(\rho)^2 ds^2 + \sn(\rho)^2 g_{\mathbb{S}^{n-2}}\end{equation} and the coordinate vector field $\pr_s = \frac{\pr}{\pr s}$ (suitably extended) generates an isometry of $M$. We recall from \cite{NZ} that $\pr_s = \cs(\rho)^2\nabla s$ is a Killing field.  In particular, we have $g(\nabla_X \pr_s, X)=0$ for any vector field $X$.

\section{Fixed-centre monotonicity and the area estimate through the origin}

In this section, we revisit classical monotonicity formulae in concentric balls in space forms. We also review their proofs, as some are not explicitly stated in the literature, and they provide motivation for our novel monotonicity formulae. The most classical monotonicity formula is the area monotonicity for minimal submanifolds in $\mathbb{R}^n$. 

In addition to area monotonicity (which also holds in $\mathbb{H}^n$ but not $\mathbb{S}^n$), there is a weighted area monotonicity and a related weighted boundary area monotonicity, both of which hold in \textit{all} of the space forms. Each of these monotone quantities is closely related to the existence of a vector field satisfying a certain divergence lower bound. 

The various monotonicity formula in Theorem \ref{thm:weighted-monotonicity} can be viewed as generalizations of the monotonicity formulae in Theorems \ref{thm:classic-monotonicity} and \ref{thm:classic-weighted-monotonicity}.

To begin, we recall the important vector fields 
\begin{equation}
W_0 := \frac{1}{A'(r)} \nabla r  = \nabla G(r)
\end{equation}
and 
\begin{equation}
W_1 := \frac{A(r)}{A'(r)} \nabla r = A(r) \nabla G(r).
\end{equation}
In the following, for any vector field $W$ we write $W = W^\top + W^\perp$ to denote the tangential and orthogonal components of $W$ along $\Sigma$. Correspondingly, let $\nabla^\top r = (\nabla r)^\top$ and $\nabla^\perp r = (\nabla r)^\perp$. We recall from \cite{NZ}:

\begin{proposition} \label{divergences}
Given a $k$-plane $S \subset T_xM$, where $0 < r(x) < \frac{1}{2}\mathrm{diam}(M)$, we have
\begin{equation}
\mathrm{div}_{S}W_0(x)  = k \frac{1}{A'(r)}\ct(r) |\nabla^\perp r|^2, 
\end{equation}
and 
\begin{equation}
\mathrm{div}_{S}W_1(x) =  |\nabla^\top r|^2 + k\frac{A(r)}{A'(r)} \ct(r)|\nabla^\perp r|^2.
\end{equation}
\end{proposition} 

Let 
\begin{equation}
a(r):=k  \frac{A(r)}{A'(r)}\ct(r).
\end{equation}
Note that if $a(r)=k  \frac{A(r)}{A'(r)}\ct(r)$, then $\big(\sn(r)^k(1-a(r))\big)' =  \curv k\sn(r) A(r)$ and $a(0) = 1$. Hence $\mathrm{sign}(1-a(r)) = \kappa$. Thus, $a(r) \geq 1$ if $M = \mathbb{H}$; $a(r) \equiv 1$ if $M = \mathbb{R}^n$, and $a(r) \leq 1$ if $M = \mathbb{S}^n$. Moreover, as long as $\cs(r) \geq 0$, one also has $a(r) \geq 0$. Consequently, (because $|\nabla^\top r|^2 + |\nabla^\perp r|^2 = 1$) we have $\div_S W_1 \geq 1$ if $M=\mathbb{H}^n$; $\div_S W_1 =1$ if $M=\mathbb{R}^n$; and $1\geq \div_S W_1 \geq |\nabla^\top r|^2$ if $M=\mathbb{S}^n$ (and $\cs(r) \geq 0$). 

\subsection{The classical monotoncity formulae}

The lower bounds on $\div_S W_1$ may be used to show the following monotonicity formulae. The first (and most well-known) is area-monotonicity, which holds when $M = \mathbb{R}^n$ or when $M = \mathbb{H}^n$ (see \cite{And82} for $\mathbb{H}^n$, \cite{Si83} for instance for $\mathbb{R}^n$):

\begin{theorem}
\label{thm:classic-monotonicity}
Suppose $M \in \{\mathbb{H}^n, \mathbb{R}^n\}$. Let $\Sigma \subset B^n_R$ be a $k$-dimensional minimal submanifold in a geodesic ball of radius $R \in (0, \frac{1}{2}\mathrm{diam}(M))$ with $\partial \Sigma \subset \partial B^n_R$. Define
\[
Q_A(t) := \frac{|\Sigma \cap B^n_t|}{|B^k_t|}.
\]
Then $t \mapsto Q_A(t)$ is monotone increasing for $t \in (0, R)$ and is constant if and only if $\Sigma$ is a totally geodesic disk.
\end{theorem}

It is not known if area-monotonicity also holds when $M = \mathbb{S}^n$. However, a closely related weighted monotonicity was proven in \cite{GS87} for the sphere. In fact, this weighted monotonicity works for all of the space forms and was a starting point for our investigation into a weighted monotonicity formula in the prescribed point problem. 

\begin{theorem}
\label{thm:classic-weighted-monotonicity}
Suppose $M \in \{\mathbb{H}^n, \mathbb{R}^n, \mathbb{S}^n\}$. Let $\Sigma \subset B^n_R$ be a $k$-dimensional minimal submanifold in a geodesic ball of radius $R \in (0, \frac{1}{2}\mathrm{diam}(M))$ with $\partial \Sigma \subset \partial B^n_R$. Define
\[
Q_I(t) :=  \frac{1}{|B^k_t|} \int_{\Sigma \cap B^n_t} |\nabla^\top r|^2. 
\]
Then $t \mapsto Q_I(t)$ is monotone increasing for $t \in (0, R)$ and is constant if and only if $\Sigma$ is a totally geodesic disk.
\end{theorem}

\begin{proof}[Proofs of Theorem \ref{thm:classic-monotonicity} and \ref{thm:weighted-monotonicity}]
By the co-area formula, the divergence theorem, and minimality of $\Sigma$, 
\begin{equation}
\label{eq:classic-monotonicity}
\frac{A(t)}{A'(t)} \frac{d}{dt} \int_{\Sigma \cap B^n_t} |\nabla^\top r|^2 = \frac{A(t)}{A'(t)}\int_{\Sigma \cap \partial B^n_t} |\nabla^\top r| = \int_{\Sigma \cap \partial B^n_t} \langle W_1, \nu \rangle = \int_{\Sigma \cap B^n_t} \mathrm{div}_{\Sigma}W_1.
\end{equation}
It follows by direct computation that
\[
Q_I'(t) = \frac{1}{|B^k_t|}\frac{A'(t)}{A(t)} \int_{\Sigma \cap B^n_t} a(r) |\nabla^\top r|^2 \geq 0. 
\]
A similar computation yields 
\[
Q_A'(t) = \frac{1}{|B^k_t|} \left(\int_{\Sigma \cap \partial B^n_t} \frac{|\nabla^\perp r|^2}{|\nabla^\top r|} + \frac{A'(t)}{A(t)} \int_{\Sigma \cap B^n_t} \big(a(r) -1\big)|\nabla^\top r|^2\right),
\]
which is evidently nonnegative if $a(r) \geq 1$ (hence for $M = \mathbb{R}^n$ and $M = \mathbb{H}^n$).
\end{proof}

\subsection{Boundary monotonicity}

Similarly, so long as $\cs(r) \geq 0$, one has the lower bound $\div_S W_0 \geq 0$. This can be used to derive another monotone quantity along the boundaries of centred geodesic balls. This monotonicity was essentially observed by Choe and Gulliver in \cite{CG92}. 

\begin{theorem}
\label{thm:bdry-monotonicity}
Suppose $M \in \{\mathbb{H}^n, \mathbb{R}^n, \mathbb{S}^n\}$. Let $\Sigma \subset B^n_R$ be a $k$-dimensional minimal submanifold in a geodesic ball of radius $R \in (0, \frac{1}{2}\mathrm{diam}(M))$ with $\partial \Sigma \subset \partial B^n_R$. Define \[Q_\pr (t):= \frac{1}{|\pr B^k_t|} \int_{\Sigma \cap \pr B^n_t} |\nabla^\top r|.\] Then $t \mapsto Q_\pr(t)$ is monotone increasing for $t \in (0, R)$ and is constant if and only if $\Sigma$ is a totally geodesic disk.
\end{theorem}

\begin{proof}
Let $\nu = \frac{\nabla^\top r}{|\nabla^\top r|}$ be the outward-pointing conormal of $\Sigma \cap \pr B^n_t$ in $\Sigma$. Recall that $|\pr B^k_t| = |\mathbb{S}^{k-1}| A'(t)$, so that 
\[ Q_\pr(t) = \frac{1}{|\mathbb{S}^{k-1}|} \int_{\Sigma\cap \pr B^n_t} \langle W_0, \nu\rangle. \]
Then for $0<s<t<R$, since $\Sigma$ is minimal the divergence theorem gives 
\[
\begin{split}
Q_\pr(t) - Q_\pr(s) &= \frac{1}{|\mathbb{S}^{k-1}|} \int_{\Sigma \cap (B^n_t\setminus B^n_s)} \div_\Sigma W_0^\top = \frac{1}{|\mathbb{S}^{k-1}|} \int_{\Sigma \cap (B^n_t\setminus B^n_s)} \div_\Sigma W_0
\\& \geq \frac{k}{|\mathbb{S}^{k-1}|} \int_{\Sigma \cap (B^n_t\setminus B^n_s)} \frac{\ct(r)}{A'(r)}|\nabla^\perp r|^2 \geq 0.
\end{split}
\]
\end{proof}

\subsection{Comparison of monotone quantities}

The calculation of (\ref{eq:classic-monotonicity}) implies that
\[
Q_{\partial}(t) = Q_I(t) +\frac{1}{A(t)} \int_{\Sigma \cap B^n_t}a(r)|\nabla^\perp r|^2. 
\]
An immediate consequence is that $Q_{\partial}(t) \geq Q_I(t)$ (as $R<\frac{1}{2}\diam(M)$, we always have $a(r)\geq 0$). Using the sign of $1-a(r)$, we conclude that for $t \in (0, R)$, we have
\[
\begin{cases}Q_{\partial}(t)  \geq Q_A(t) \geq Q_I(t), & M = \mathbb{H}^n \\ Q_{\partial}(t) = Q_A(t) \geq Q_I(t), & M = \mathbb{R}^n \\Q_A(t) \geq  Q_{\partial}(t) \geq Q_I(t), & M = \mathbb{S}^n \end{cases}. 
\]

\subsection{Area estimate at the origin}

Assume that $o\in\Sigma$. As $\Sigma$ is smooth, it follows that $|\nabla^\top r|\to 1$ as we approach $o$. This implies that
\[
Q_{\partial}(0) = Q_A(0) = Q_I(0) = \Theta(\Sigma, o) \geq 1. 
\]

If $\Sigma$ contains the origin $o$, then the classical area estimate for minimal submanifolds through the centre of the ball states: 

\begin{corollary}\label{classic-monotonicity}
Let $M,R,\Sigma$ be as above. If $o \in \Sigma$, then 
\[
|\Sigma| \geq |B^k_R|. 
\] 
Moreover, equality holds if and only if $\Sigma$ is a totally geodesic disk.
\end{corollary}

This corollary may be deduced from any of the monotonicity formulae above as follows: 

\begin{itemize} 
\item Using the monotonicity of $Q_I(t)$ (or $Q_A(t)$, if available), we have for every $t \in (0, R]$ that
\[
Q_A(t) \geq Q_I(t) \geq Q_I(0) = \Theta(\Sigma, o). 
\]
This means for every $t \in (0, R]$
\[
|\Sigma \cap B^n_t| \geq A(t) |\mathbb{S}^{k-1}|\Theta(\Sigma, o) = |B^n_t| \Theta(\Sigma, o).
\]
Taking $t = R$ is the area estimate. 
\item Using the monotonicity of $Q_\partial(t)$, we have for every $t \in (0, R]$ that
\[
\frac{|\Sigma \cap \partial B^n_t|}{|\partial B^k_t|} \geq Q_\partial (t) \geq Q_\partial(0) = \Theta(\Sigma, o).
\]
This means for every $t \in (0, R]$, 
\[
|\Sigma \cap \partial B^n_t| \geq A'(t) |\mathbb{S}^{k-1}|\Theta(\Sigma, o) = |\partial B^n_t| \Theta(\Sigma, o)
\]
Integrating this implies the area estimate for every $t \in (0, R]$. 
\end{itemize}

Corollary \ref{classic-monotonicity} also follows directly from the divergence theorem applied to the linear combination $W_1- A(R)W_0$. In fact, using this direct method yields the area estimate for any possible $R$, whereas the monotonicity formulae are restricted to $R<\frac{1}{2}\diam(M)$. 

\section{Moving-centre monotonicity formulae} 

In this section, we present our novel monotonicity formulae, and also show that they may be used to deduce sharp area bounds for minimal submanifolds through a prescribed point. 

\begin{lemma}
\label{lem:weighted-monotonicity}
Let $M \in \{\mathbb{H}^n, \mathbb{R}^n, \mathbb{S}^n\}$. Consider a smooth function $f$ on $M$ and its sublevel sets $E_t  = \{f\leq t\}$. 
Let $\Sigma$ be a $k$-dimensional minimal submanifold in $E_1$ with $\partial \Sigma \subset \partial E_1$. Assume that there is a family of functions $w_t$ on $\Sigma$ and a family of vector fields $W_t$ on $M$ such that:
\begin{enumerate}
\item $w_t - \langle W_t, \frac{\nabla^\top f}{f}\rangle \geq 0$ on $\Sigma\cap \{f=t\}$;
\item $\div_\Sigma W_t \geq w_t - t \pr_t w_t$ on $\Sigma \cap \{f\leq t\}$.
\end{enumerate}

Then the quantity
\begin{equation}
\frac{\int_{\Sigma\cap E_t}w_t }{t}
\end{equation}
is monotone increasing for $t \leq 1$. 
\end{lemma}
\begin{proof}
By the coarea formula, we have

\[\frac{d}{dt}\left(\frac{1}{t} \int_{\Sigma\cap E_t} w_t\right) = \frac{1}{t} \int_{\Sigma \cap \pr E_t} \frac{w_t}{|\nabla^\top f|} - \frac{1}{t^2}\int_{\Sigma\cap E_t} w_t +\frac{1}{t} \int_{\Sigma\cap E_t} \pr_t w_t.\] 
Note that the outer unit conormal of $\Sigma\cap \pr E_t$ in $\Sigma \cap E_t$ is $\frac{\nabla^\top f}{|\nabla^\top f|}$. 
Using assumption (2) and the divergence theorem, we therefore have
\[\frac{d}{dt}\left(\frac{1}{t} \int_{\Sigma\cap E_t} w_t\right) \geq \frac{1}{t} \int_{\Sigma \cap \pr E_t} \frac{1}{|\nabla^\top f|} \left( w_t - \langle W_t , \frac{\nabla^\top f}{f}\rangle \right).\] The right hand side is nonnegative by assumption (1), which completes the proof. 
\end{proof}

\begin{definition}
Let $y\in B^n_R$, $s_y=s(y) >0$. We recall the following two definitions from \cite{NZ}:
\begin{equation}
\label{eq:def-u}
u(s) := 
\begin{cases}
\cs^{-1}\Big(\frac{\cs(s - s_y)\cs(R)}{\cs(s)}\Big), & M= \{\mathbb{H}^n, \mathbb{S}^n\}\\
\big(R^2 + (s_y-s)^2  - s^2\big)^{\frac{1}{2}}, & M = \mathbb{R}^n
\end{cases},
\end{equation}
\begin{equation}
F(s) := A'(u(s))u'(s) \cs(s-s_y)^2, 
\end{equation}
and set $u_s(x) = u(s(x))$. We recall the shorthand $u_s'(x):= u'(s(x))$ and note that $u_s' \leq 0$ (as can be directly checked). Consequently, $F(s) \leq 0$. Further define $f = \frac{A(r_y)}{A(u_s)}$ and its sublevel sets $E_t = \{f\leq t\}$, so that $E_1$ is precisely $B^n_R$ (at least assuming $\cs(R + s_y) \geq 0$). 

\begin{figure}[h]
\centering
\includegraphics[scale=0.75]{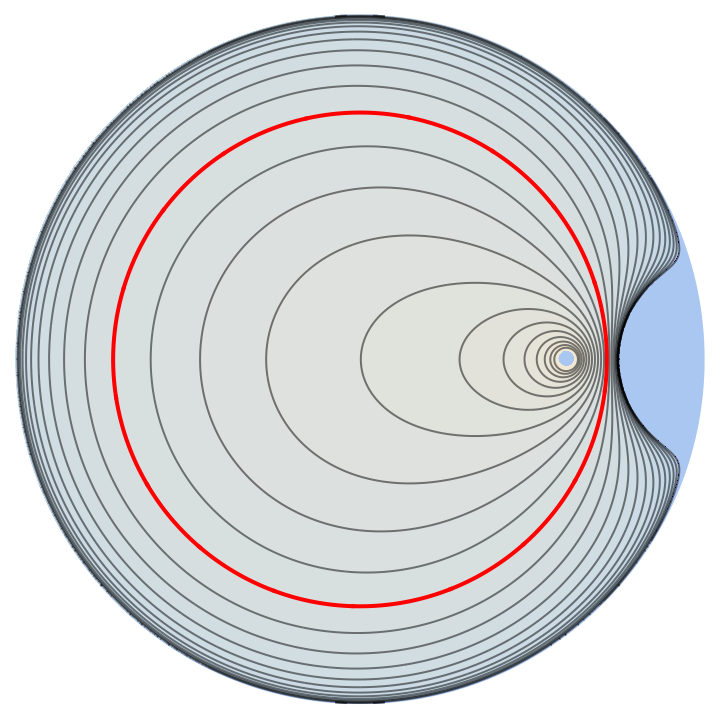}
\caption{Representation of the sets $E_t$, $t \in (0, \infty)$, in the Poincar\'e disk model of $\mathbb{H}^n$. As $t\to 0$, the sets $E_t$ approximate balls around $y$. Then $E_1 = B^n_R$, which is represented by the red (thicker) circle in the figure above. Finally, as $t\to \infty$, the sets cover a halfspace of $\mathbb{H}^n$ (with totally geodesic boundary). }
\label{figure:Et}
\end{figure}

We now set \begin{equation}
\tilde{F}_t(s) :=  t^2 \frac{A(u(s))}{A'(A^{-1}(tA(u(s))))^2 \cs(A^{-1}(tA(u(s))))^2 } F(s)
\end{equation}
and for $i,j \in \{0,1\}$ the define the weights
\begin{equation}\label{eq:weight}
w_{t, i, j} = |\nabla^\top r_y|^2 + i |\nabla^\perp r_y|^2 - j \tilde{F}_t(s) \frac{A'(u_s)u'_s}{A(u_s)} \cs(\rho)^2 |\nabla^\top s|^2.
\end{equation}

\end{definition} 

\begin{remark}
\label{rmk:F'}
It was shown in \cite{NZ} that $F'(s)\geq 0$ is equivalent to $k \cs(u(s))^2 \geq 2$. In particular, this condition is always satisfied if $M\in\{\mathbb{H}^n,\mathbb{R}^n\}$.
\end{remark}

\begin{theorem}
\label{thm:weighted-monotonicity}
Suppose that either
\begin{enumerate}
\item $M = \mathbb{R}^n$, $i\in\{0,1\}$, $j\in\{0,1\}$;
\item $M=\mathbb{H}^n$, $i\in \{0,1\}$, $j=1$;
\item $M=\mathbb{S}^n$, $i=0$, $j=1$. 
\end{enumerate}
 Further suppose that $R<\frac{1}{2}\diam(M),y$ are such that $F'(s)\geq 0$ for $|s|\leq R$. 

Let $\Sigma$ be a $k$-dimensional minimal submanifold in $B^n_R$ with $\partial \Sigma \subset \partial B^n_R$, and let $E_t$ be as above. Let $\Sigma_0$ be a totally geodesic disk orthogonal to the geodesic $\gamma$ containing $o$ and $y$.

Then the quantity \begin{equation}\label{eq:monotone-quantity}Q_{i,j}(t):=\frac{\int_{\Sigma\cap E_t} w_{t,i,j} }{|\Sigma_0 \cap E_t|} = \frac{1}{A(\underline{r}(y))|\mathbb{S}^{k-1}|}  \frac{\int_{\Sigma\cap E_t} w_{t,i,j} }{t}\end{equation} is monotone increasing for $t\in [0,1]$, and is constant if and only if $\Sigma$ is a totally geodesic disk orthogonal to the geodesic $\gamma$. 
\end{theorem}

\begin{remark}
When $y=0$, this is equivalent to the classical monotonicity Theorem \ref{thm:classic-monotonicity}. When $M=\mathbb{R}^n$, the monotonicity of $Q_{1,0}$ is equivalent to the moving-centre monotonicity formula proven in \cite{Zhu18}. The other monotonicity formulae are new, even in the Euclidean setting. We remark though that the sets $E_t$ are the same family of sets used in \cite{Zhu18}. 
\end{remark}

\begin{remark}
Each monotone quantity can be used to obtain an excess-type estimate for the area by integrating the derivative. It follows from the proof below that 
\begin{align}\label{eq:excess}
Q_{i,j}'(t)&=  \frac{1}{t} \int_{\Sigma \cap \pr E_t} \frac{1}{|\nabla^\top f|}\left(  i|\nabla^\perp r_y|^2 +(1- j)\tilde{F}_t(s)  \frac{A'(u_s)u'_s}{A(u_s)}  \cs(\rho)^2 |\nabla^\top s|^2\right)  \\
\nonumber &\quad + \frac{1}{t^2}\left(\int_{\Sigma\cap E_t}  (a(r_y)-i)|\nabla^\perp r_y|^2 + j \tilde{F}_t(s)\frac{F'(s)}{F(s)}\cs(\rho)^2 |\nabla^\top s|^2 + (1-j)\pr_s \tilde{F}_t(s) \cs(\rho)^2 |\nabla^\top s|^2\right).
\end{align}
\end{remark}

\begin{proof}
Note that on $\Sigma_0$, we have $u_s \equiv \underline{r}(y)$, so $f|_{\Sigma_0} = \frac{A(r_y)}{A(\underline{r}(y))}$. It follows that $\Sigma_0 \cap \{f\leq t\}$ is precisely a totally geodesic disk with area $tA(\underline{r}(y)) |\mathbb{S}^{k-1}|$. This establishes the equality (\ref{eq:monotone-quantity}). 

The strategy is now to apply Lemma \ref{lem:weighted-monotonicity} with the vector field \[W_t = \frac{A(r_y)}{A'(r_y)} \nabla r_y + \tilde{F}_t(s) \pr_s.\] 
We need only verify conditions (1) and (2) of that lemma:

For (1), note that \[\frac{\nabla f}{f} = \frac{A'(r_y)}{A(r_y)}\nabla r_y - \frac{A'(u_s)}{A(u_s)} u'_s \nabla s.\]
Therefore 
\[
\begin{split}
w_{t,i,j} - \langle W_t, \frac{\nabla^\top f}{f}\rangle = w_{t,i,j} &- |\nabla^\top r_y|^2  - \frac{A'(r_y)}{A(r_y)}\tilde{F}_t(s) \langle \nabla^\top r_y, \pr_s\rangle \\
&+\frac{A(r_y)}{A'(r_y)} \frac{A'(u_s)u'_s}{A(u_s)} \langle \nabla^\top r_y ,\nabla^\top s\rangle  + \tilde{F}_t(s)  \frac{A'(u_s)u'_s}{A(u_s)}  \cs(\rho)^2 |\nabla^\top s|^2
\end{split}
\]
When $f= \frac{A(r_y)}{A(u_s)}=t$, we have $\tilde{F}_t(s)= \frac{A(r_y)^2}{A(u_s)} \frac{A'(u_s) u'_s \cs(s-s_y)^2}{A'(r_y)^2\cs(r_y)^2}$, which means the middle terms precisely cancel. As $\tilde{F}_t(s)\leq 0$, $u'_s\leq 0$, it follows that on $\{f=t\}$, we have
\[ w_{t,i,j} - \langle W_t, \frac{\nabla^\top f}{f}\rangle= i|\nabla^\perp r_y|^2 +(1- j)\tilde{F}_t(s)  \frac{A'(u_s)u'_s}{A(u_s)}  \cs(\rho)^2 |\nabla^\top s|^2  \geq0\]

For (2), we calculate
\begin{equation}
\div_\Sigma W_t = |\nabla^\top r_y|^2 + a(r_y) |\nabla^\perp r_y|^2 + \pr_s \tilde{F}_t(s) \cs(\rho)^2 |\nabla^\top s|^2. 
\end{equation}
On the other hand, \[w_{t,i,j} - t \pr_t w_{t,i,j} = |\nabla^\top r_y|^2 + i |\nabla^\perp r_y|^2 - j (\tilde{F}_t(s)-t\pr_t \tilde{F}_t(s)) \frac{A'(u_s)u'_s}{A(u_s)} \cs(\rho)^2 |\nabla^\top s|^2.\]

When $\kappa \leq 0$, we have $a(r_y)\geq 1\geq i$. When $\kappa =1$, we still have $a(r_y)\geq 0$, and $i=0$ by assumption. In either case, we have $|\nabla^\top r_y|^2 + a(r_y) |\nabla^\perp r_y|^2 \geq |\nabla^\top r_y|^2 + i |\nabla^\perp r_y|^2$, and it remains to verify that 
\[
\pr_s \tilde{F}_t(s) + j(\tilde{F}_t(s)-t\pr_t \tilde{F}_t(s)) \frac{A'(u_s)u'_s}{A(u_s)} \geq 0,
\]
or equivalently 
\begin{equation}
\label{eq:div-condition}
\frac{\pr_s \tilde{F}_t(s)}{\tilde{F}_t(s)} + j\left(1-t \frac{\pr_t \tilde{F}_t(s)}{\tilde{F}_t(s)}\right) \frac{A'(u_s)u'_s}{A(u_s)} \leq 0.
\end{equation}

Henceforth, for convenience, we will write $u=u(s)$. We calculate 
\[\pr_s A^{-1}(tA(u)) =  \frac{tA'(u(s))u'(s)}{A'(A^{-1}(tA(u)))} ,\]

\[\pr_t A^{-1}(tA(u)) = \frac{A(u(s))}{A'(A^{-1}(tA(u)))}.\] 
Then \[ \frac{\pr_s \tilde{F}_t(s)}{\tilde{F}_t(s)} = \frac{F'(s)}{F(s)} + \frac{A'(u) u'}{A(u)} - 2 \frac{A''(A^{-1}(tA(u)))}{A'(A^{-1}(tA(u)))^2} tA'(u)u' + 2\kappa \frac{\tn(A^{-1}(tA(u)))}{A'(A^{-1}(tA(u)))}tA'(u)u',\] 

\[\frac{\pr_t \tilde{F}_t(s)}{\tilde{F}_t(s)} = \frac{2}{t} - 2 \frac{A''(A^{-1}(tA(u)))}{A'(A^{-1}(tA(u)))^2} A(u) + 2\kappa \frac{\tn(A^{-1}(tA(u)))}{A'(A^{-1}(tA(u)))}A(u).\]
It follows that, for any $\kappa$, we have 
\[\frac{\pr_s \tilde{F}_t(s)}{\tilde{F}_t(s)} + \left(1-t \frac{\pr_t \tilde{F}_t(s)}{\tilde{F}_t(s)}\right) \frac{A'(u_s)u'_s}{A(u_s)} = \frac{F'(s)}{F(s)} \leq0.\]

This establishes (\ref{eq:div-condition}) whenever $j=1$. In the remaining case, we have $\kappa=0, j=0$. Then $A^{-1}(tA(u))= t^\frac{1}{k}u$, and $\tilde{F}_t(s) = t^2 \frac{u^k}{k t^\frac{2(k-1)}{k} u^{2(k-1)}} u^{k-1} u' = - \frac{s_y}{k} t^\frac{2}{k} $. In particular $\pr_s \tilde{F}_t(s)=0$, which again establishes (\ref{eq:div-condition}). 

The monotonicity now follows from Lemma \ref{lem:weighted-monotonicity}. Finally, if $\frac{1}{t} \int_{\Sigma\cap E_t} w_{t,i,j}$ is constant, then we must have equality in all the inequalities above. It is clear that this may only occur if $\nabla^\perp r_y=0$ and (if $y\neq o$) $\nabla^\top s=0$ on $\Sigma$, which implies that $\Sigma$ is a totally geodesic disk orthogonal to $\gamma$. 
\end{proof}

\begin{corollary}
\label{cor:area-estimate}
Let $M\in \{\mathbb{H}^n, \mathbb{R}^n, \mathbb{S}^n\}$ and suppose that $R,y$ are such that $F'(s)\geq 0$ for $|s|\leq R$. Let $\Sigma$ be a $k$-dimensional minimal submanifold in $B^n_R$ with $\partial \Sigma \subset \partial B^n_R$, and $y\in \Sigma$. 

Then $|\Sigma| \geq |B^k_{\underline{r}(y)}|$, with equality if and only if $\Sigma$ is a totally geodesic disk orthogonal to $y$. 
\end{corollary}
\begin{proof}
Theorem \ref{thm:weighted-monotonicity} gives that \[ \frac{\int_\Sigma w_{1,i,j}}{|\Sigma_0\cap E_1|}  \geq \frac{1}{A(\underline{r}(y))|\mathbb{S}^{k-1}|} \lim_{t\to0} \frac{\int_{\Sigma\cap E_t} w_{t,i,j} }{t}\] for some suitable $i,j$. 
Now we always have $\tilde{F}_t(s) \leq 0$ and $u_s' \leq 0$, hence it is clear from \eqref{eq:weight} that \[w_{t,i,j} \leq 1.\] Moreover, $E_1 = B^n_R$, so $|\Sigma_0\cap E_1| = |B^k_{\underline{r}(y)}|.$ 
On the other hand, since $u_s$ is bounded away from zero on $B^n_R$, we see that as $f\to0$, so too does $r_y\to 0$. Then as $u_s(y)= \underline{r}(y)$, we find that \[f \sim \frac{A(r_y)}{A(\underline{r}(y))}.\]

For small $t$, it follows that $E_t$ approximates a ball of radius $A^{-1}(tA(\underline{r}(y)))$, and hence 
 $\Sigma\cap \{f\leq t\}$ approximates a disk of area $t A(\underline{r}(y))|\mathbb{S}^{k-1}|$. Moreover, as $t\to 0$, we have \[\tilde{F}_t \sim t^\frac{2}{k} \frac{A(\underline{r}(y))}{ \underline{r}(y)^{2(k-1)}} F(s_y) \to0,  \] and $\nabla^\top r_y \to \nabla r_y$. In particular, $w_{t,i,j} \to 1$, and \[ \frac{1}{A(\underline{r}(y))|\mathbb{S}^{k-1}|} \lim_{t\to0} \frac{\int_{\Sigma\cap E_t} w_{t,i,j} }{t} =  \Theta(\Sigma, y)\geq 1.\] This completes the proof. 

\end{proof}

\bibliographystyle{amsalpha}
\bibliography{mc-monotonicity-space-forms}

\end{document}